\documentclass{article}
\usepackage[utf8]{inputenc}
\usepackage[english]{babel}
\usepackage{amsmath,amsfonts}
\usepackage{amsthm}
\newtheorem{thm}{Theorem}

\newtheorem{lem}{Lemma}
\newtheorem{cor}{Corollary}
\theoremstyle{remark}

\theoremstyle{definition}

\def\Real{\mathop{\rm Re}}
\def\Imag{\mathop{\rm Im}}
\def\tr{\mathop{\rm Tr}}
\def\Arcosh{\mathop{\rm Arcosh}}

\title{Eigenvalue Bounds for Perturbed Periodic Dirac Operators}
\author{Ghada Shuker Jameel$^{1,2}$, Karl Michael Schmidt$^1$
\thanks{corresponding author}
\\
\\ {\it $^1$School of Mathematics, Cardiff University, Wales, UK\/}
\\ {\it $^2$Department of Mathematics, College of Education\/}
\\ {\it for Pure Science, University of Mosul, Iraq\/}
}

\date{}

\begin{document}

\maketitle
\begin{abstract}
We characterise regions in the complex plane that contain all non-embedded
eigenvalues of a perturbed periodic Dirac operator on the real line with
real-valued periodic potential and a generally non-symmetric matrix-valued
perturbation $V$. We
show that the eigenvalues are located close to the end-points of the
spectral bands for small $V\in L^1(\mathbb R)^{2\times 2}$, but only close
to the spectral bands as a whole for small $V\in L^p(\mathbb R)^{2\times 2}$, $p > 1$.
As auxiliary results, we prove the relative compactness of matrix
multiplication operators in $L^{2p}(\mathbb R)^{2\times 2}$ with respect to
the periodic operator under minimal hypotheses, and find the asymptotic solution
of the Dirac equation on a finite interval for spectral parameters with large
imaginary part.
\\
{\bf Keywords:}
Non-selfadjoint operator; periodic Dirac system; eigenvalue enclosure
\\
{\bf 2020 MSC:} 47B28; 34L40, 47A55, 81Q15
\end{abstract}

\section{Introduction}
In the present paper, we consider the one-dimensional perturbed periodic Dirac
operator
\begin{equation*}
H = -i\,\sigma_2\,\frac d{dx} + m\,\sigma_3 + q(x) + V(x) \qquad (x\in\mathbb R),
\end{equation*}
where 
$\sigma_2$
and
$\sigma_3$
are Pauli matrices (see equation (\ref{eq:Pauli}) below), $m \ge 0$ is the particle mass,
$q : \mathbb R \rightarrow \mathbb R$ is a periodic potential and
$V : \mathbb R \rightarrow \mathbb C^{2\times 2}$ is a matrix-valued perturbation. Although the unperturbed periodic operator
\begin{equation*}
H_0 = -i\,\sigma_2\,\frac d{dx} + m\,\sigma_3 + q(x)\qquad (x\in\mathbb R),
\end{equation*}
is a self-adjoint operator in $L^2(\mathbb R)^2$, the operator $H$ is not
self-adjoint in general as we do not assume that the matrix multiplication
operator $V$ is symmetric.
We assume that $V$ is bounded and that $V\in L^p(\mathbb R)^{2\times 2}$ for
some $p \ge 1$. Then $H$ has the same essential spectrum as $H_0$, consisting of closed intervals on the real line (spectral bands), generally separated by spectral gaps, but may in addition have
discrete eigenvalues in the complex plane
(see Theorem \ref{thm:compres} below).

Our aim is to find a priori enclosures for these eigenvalues, i.e. regions
characterised in terms of the properties of the unperturbed periodic Dirac
equation and the $p$-norm of $V$ which contain all (non-embedded) eigenvalues
of $H$.
In the absence of a periodic background potential, $q = 0$, \cite{CLT} proved
that, for $V\in L^1(\mathbb R)^{2\times 2}$ with $\|V\|_1 < 1$, the non-embedded eigenvalues of $H$ lie within circles around (but not
centred at) the points $\pm m$, the end-points of the two intervals of
essential spectrum $\sigma_e(H) = (-\infty, -m] \cup [m, \infty)$.
The radii of the circles tend to 0 as the 1-norm of $V$ tends to 0, showing that
when a coupling parameter $\epsilon$ is employed, the eigenvalues of
$H_0 + \epsilon V$ emanate from the points $\pm m$ only as $\epsilon$ increases
from 0.

In the present study, we extend this observation to the case where a periodic
background potential $q$ is present and allow $V$ to be $p$-integrable with
$p \ge 1$.
Our main eigenvalue exclusion result (Theorem \ref{thm:eet}) states that a
complex number $\lambda$ outside the essential spectrum of $H$ cannot be an
eigenvalue of $H$ if
the $p$-norm of $V$ (defined in equation (\ref{eq:Vnorm}) below) satisfies
the inequality
$\|V\|_p < F_p(\lambda)$, where $F_p$ is some non-negative function determined completely
in terms of solution properties of the unperturbed periodic equation.
From our results, the following picture emerges.
For $p = 1$, $F_1$ is bounded above by 1 and in fact tends to 1 as
$|\Imag\lambda| \rightarrow \infty$ (see Theorems \ref{thm:Masymp}, \ref{thm:phipmasymp}), so its level sets for levels $< 1$ lie in neighbourhoods of the
real line.
Moreover, $F_1$ tends to zero exactly at the
end-points of spectral bands (Theorems \ref{thm:gammabound}, \ref{thm:MisI}, \ref{thm:MisnotI}).
This means that for small $\|V\|_1 < 1$, the eigenvalues are confined to small neighbourhoods of the end-points of spectral bands and, when a coupling parameter is
applied, will emerge from these end-points only.
This behaviour appears to be a natural analogue to that observed in \cite{CLT} and \cite{CS}.

However, for $p > 1$, $F_p(\lambda)$ grows beyond all bounds as $|\Imag\lambda|\rightarrow\infty$.
Therefore the level sets of $F_p$ will be in neighbourhoods of the real
line for all positive levels, and we get eigenvalue enclosure regions for
any size of $\|V\|_p$.
However, $F_p$ tends to 0 at all points of the essential spectrum of $H$,
which means that for small $\|V\|_p$ the eigenvalues are confined to small
neighbourhoods of the whole spectral bands.
Although we do not show the actual appearance of eigenvalues in such position
here, this opens up the possibility of eigenvalues approaching (or, with a
coupling parameter, emerging from) any point of the essential spectrum of $H$,
similar to the behaviour observed in \cite{Boegli} for Schr\"odinger operators.

We mention that in the recent study \cite{BKN}, a detailed spectral analysis of
the different, but related Dirac operator where, instead of a real periodic
potential, $q$ is a purely imaginary jump potential was performed.

The present paper is structured as follows.
In Section \ref{sec:periodic} we summarise the relevant results from Floquet
theory of the periodic Dirac equation, describing in particular the definition
of the complex quasimomentum used in this paper.
We also give a formula for the resolvent operator of $H_0$ and show that it
is a bounded linear operator not only in $L^2(\mathbb R)^2$, but also between a
dual pair of non-Hilbert Lebesgue spaces (Theorem \ref{thm:resolvent}).
In Section \ref{sec:exclusion} we first prove that $H$ has the same essential
spectrum (for all five usual definitions for a non-selfadjoint operator) as
$H_0$ and only discrete eigenvalues besides (Theorem \ref{thm:compres}).
A key part of the proof is the observation that the operator of multiplication
with a matrix-valued function in $L^{2p}(\mathbb R)$ is $H_0$-relatively
compact (Lemma \ref{lem:compres}), for which we provide a proof as it is not easily
found in the literature in this generality, with locally integrable $q$, and
hence may be of independent interest.
We then proceed to the main eigenvalue exclusion theorem (Theorem \ref{thm:eet})
already described above.
In Section \ref{sec:near}, we show that the function determining the
exclusion criterion for $p = 1$ tends to zero exactly at the end-points of the spectral
bands.
Finally, in Section \ref{sec:far}, we show that this function tends to 1 as
$\Imag\lambda \rightarrow \infty$. This result is based on the general
asymptotics of the fundamental system of the Dirac equation on a finite
interval for this limit (Theorem \ref{thm:Philim}), which is here obtained
using a novel transformation of the Dirac equation into the pair of coupled
differential equation systems (\ref{eq:sepsys}) and may be of interest in its
own right.

As a matter of notation, we write $|w|$ for the Euclidean norm $\sqrt{|w_1|^2 + |w_2|^2}$ of vectors $w\in\mathbb C^2$.

\section{The periodic equation}\label{sec:periodic}

\noindent
Let $\Phi(\cdot, \lambda)$ be the canonical fundamental system of the periodic Dirac
equation with spectral parameter $\lambda\in\mathbb C$, i.e.\ the solution
of the (matrix) initial value problem
\begin{equation}
 -i\sigma_2 \Phi'(x, \lambda) + (m \sigma_3 + q(x))\,\Phi(x,\lambda) = \lambda\Phi(x, \lambda) \quad (x \in \mathbb R),
 \quad \Phi(0, \lambda) = \mathbb I,
\label{eq:pD}
\end{equation}
where $\mathbb I$ is the $2\times 2$ unit matrix
and $q$ is a locally integrable, real-valued, periodic function.
The qualitative behaviour of the
solutions can be studied by means of Floquet theory considering the monodromy
matrix
$M(\lambda) := \Phi(a,\lambda)$ $(\lambda\in\mathbb C)$,
where $a > 0$ is the period of $q$, see \cite{BES}. As the (Wronskian)
determinant of the monodromy matrix is equal to 1, its eigenvalues are
inverses of each other. Their positions in the complex plane can be characterised
in terms of the discriminant
$\mathfrak D(\lambda) := \tr M(\lambda)$.
The characteristic equation for $M(\lambda)$,
\begin{equation*}
\mu^2 - \mathfrak D(\lambda)\,\mu + 1 = 0,
\end{equation*}
shows that $M(\lambda)$ has two distinct eigenvalues if and only if
$\mathfrak D(\lambda) \notin \{-2, 2\}$.
In this case, either the eigenvalues lie on the unit circle and are complex
conjugates of each other (this happens when $\mathfrak D(\lambda) \in (-2, 2)$), or one
eigenvalue, $\rho(\lambda)$, lies outside, the other eigenvalue, $1/\rho(\lambda)$, lies inside the unit circle (this happens when
$\mathfrak D(\lambda) \in \mathbb C \setminus [-2, 2]$).
If $\mathfrak D(\lambda) \in \{-2, 2\}$, then either the geometric multiplicity
of the eigenvalue $\pm 1$ is 1 or $M(\lambda) = \pm \mathbb I$
(see \cite[Section 1.4]{BES}).

If $\mu$ is an eigenvalue of $M(\lambda)$ and $v \in \mathbb C^2\setminus \{0\}$
is a corresponding eigenvector, then
$u(x) := \Phi(x,\lambda)\,v$ $(x \in \mathbb R)$
is a {\it Floquet solution\/}
of the Dirac equation
\begin{equation}
-i\sigma_2 u'(x) + (m\sigma_3 + q(x))\,u(x) = \lambda\,u(x) \qquad (x \in\mathbb R);
\label{eq:Dirac}
\end{equation}
clearly $u(0) = v$. Then the function
$\varphi(x) := \mu^{-x/a}\,u(x)$ $(x\in\mathbb R)$ is $a$-periodic.
This shows that all solutions of the periodic Dirac equation are bounded
if $\mathfrak D(\lambda) \in (-2, 2)$ and that there is one Floquet solution
$u_+(\cdot,\lambda)$
exponentially small at $-\infty$ and one Floquet solution $u_-(\cdot,\lambda)$ exponentially small at
$\infty$ if $\mathfrak D(\lambda) \in \mathbb C \setminus [-2, 2]$. If
$\mathfrak D(\lambda) \in \{-2, 2\}$, then either one or all solutions are
bounded.
Hence we can deduce that $\sigma(H_0) = \{\lambda \in \mathbb C \mid \mathfrak D(\lambda) \in [-2, 2]\} \subset \mathbb R$ for the self-adjoint operator $H_0 = -i\sigma_2\frac d{dx} + m \sigma_3 + q$ (see also \cite[Theorem 4.7.1]{BES}).

The (entries of the) monodromy matrix $M$ and hence also the discriminant
$\mathfrak D$ are entire functions, cf. \cite[Theorem 1.7.2]{Eastham70}.
Since $m > 0$ and $q$ is real valued, it follows that
$\Phi(x, \overline\lambda) = \overline{\Phi(x, \lambda)}$ $(x \in \mathbb R)$ and so $M(\overline\lambda) = \overline{M(\lambda)}$ and $\mathfrak D(\overline\lambda) = \overline{\mathfrak D(\lambda)}$ for all $\lambda\in\mathbb C$.
If $\mathfrak D(\lambda) \notin [-2,2]$, let $v_+(\lambda)$ and $v_-(\lambda)$
be eigenvectors corresponding to the eigenvalues $\rho(\lambda)$ and
$1/\rho(\lambda)$ of $M(\lambda)$, respectively.
Then $\rho(\overline\lambda) = \overline{\rho(\lambda)}$ and we can
choose the eigenvectors such that $v_\pm(\overline\lambda) = \overline{v_\pm(\lambda)}$.
Therefore we focus on $\lambda$ with $\Imag\lambda \ge 0$ in the following.

The discriminant can be written in the form
\begin{align}
 \mathfrak D&(\lambda) = 2 \cos k(\lambda) a
\nonumber
\\
 &= 2 \cosh (a \Imag k(\lambda))\cos (a \Real k(\lambda)) - 2 i \sinh (a \Imag k(\lambda))\sin (a \Real k(\lambda))
\label{eq:Dk}
\end{align}
$(\lambda\in\mathbb C, \Imag \lambda \ge 0)$,
where the (continuous) function $k$
with $\Imag k(\lambda) \ge 0$ is called the complex {\it quasimomentum\/}
(see also e.g.\ \cite{KK97}).
As can be seen from equation (\ref{eq:Dk}), for $\lambda\in\mathbb R$, the quasimomentum $k(\lambda)$ is real;
it is closely
related to the rotation number (cf.\ \cite[p.43]{BES}) in the intervals where
$\mathfrak D(\lambda) \in [-2, 2]$ (stability intervals), whereas it has
constant real part $\in\pi\mathbb Z$ and positive imaginary part in the
intervals where $\mathfrak D(\lambda) \notin [-2, 2]$ (instability intervals).
More generally, for $\lambda \in\mathbb C$ such that $\Imag\lambda \ge 0$ and $\mathfrak D(\lambda) \notin [-2, 2]$, the eigenvalue of $M(\lambda)$ that lies outside
the unit circle is $\rho(\lambda) = e^{-i k(\lambda) a}$,
the other eigenvalue being $1/\rho(\lambda) = e^{i k(\lambda) a}$.
Clearly $k(\lambda) \in \mathbb R$ implies that $\mathfrak D(\lambda)\in [-2,2]$ and so $\lambda\in\mathbb R$.
We also note the following. 
\begin{lem}
Let $\lambda\in\mathbb C$, $\Imag\lambda\ge 0$. Then
\begin{equation}
\Imag k(\lambda) = \frac 1{2a} \Arcosh\left(\frac{|\mathfrak D(\lambda)|^2}4 + \sqrt{\left(1 - \frac{|\mathfrak D(\lambda)|^2}4\right)^2 + (\Imag \mathfrak D(\lambda))^2}\right).
\label{eq:Imk}
\end{equation}
In particular,
\begin{equation}
\lim\limits_{\lambda\rightarrow\lambda_0} \Imag k(\lambda) = 0
\label{eq:Imklim}
\end{equation}
if $\mathfrak D(\lambda_0) \in [-2, 2]$.
\end{lem}
\begin{proof}
If $\Imag k(\lambda) = 0$, then by equation (\ref{eq:Dk}) $\mathfrak D(\lambda) = 2 \cos a k(\lambda) \in [-2, 2]$ and the right-hand side in equation (\ref{eq:Imk}) vanishes.
If $\Imag k(\lambda) > 0$, then by equation (\ref{eq:Dk}) we find
\begin{align*}
1 &= \frac{(\Real\mathfrak D(\lambda))^2}{4 \cosh^2 (a \Imag k(\lambda))} + \frac{(\Imag\mathfrak D(\lambda))^2}{4 \sinh^2 (a \Imag k(\lambda))}
\\
&= \frac{(\Real\mathfrak D(\lambda))^2 (\cosh (2 a \Imag k(\lambda)) - 1) + (\Imag\mathfrak D(\lambda))^2 (\cosh (2 a \Imag k(\lambda)) + 1)}{2 (\cosh^2 (2 a \Imag k(\lambda)) - 1)}
\end{align*}
and hence by solving the quadratic equation
\begin{align*}
\cosh (2 a \Imag k(\lambda)) &= \frac{|\mathfrak D(\lambda)|^2}4 \pm \sqrt{1 + \frac{|\mathfrak D(\lambda)|^4}{16} - \frac{(\Real\mathfrak D(\lambda))^2 - (\Imag\mathfrak D(\lambda))^2}2}
\\
&= \frac{|\mathfrak D(\lambda)|^2}4 \pm \sqrt{\left(1 - \frac{|\mathfrak D(\lambda)|^2}4\right)^2 + (\Imag \mathfrak D(\lambda))^2}.
\end{align*}
Since
\begin{equation*}
\frac{|\mathfrak D(\lambda)|^2}4 - \sqrt{\left(1 - \frac{|\mathfrak D(\lambda)|^2}4 \right)^2 + (\Imag \mathfrak D(\lambda))^2} \le 1
\end{equation*}
and $\cosh (2 a \Imag k(\lambda)) > 1$ in the case under consideration, the
square root must have the positive sign.
\end{proof}
For $\Imag\lambda < 0$, the Floquet multiplier (eigenvalue) satisfies
\begin{equation*}
 \rho(\lambda) = \overline{\rho(\overline\lambda)}
 = \overline{e^{-i k(\overline\lambda) a}}
 = e^{-i(-\overline{k(\overline\lambda)}) a}.
\end{equation*}
This motivates the definition of the quasimomentum in the complex lower
half-plane by setting
$k(\lambda) := - \overline{k(\overline\lambda)}$ $(\lambda\in\mathbb C, \Imag\lambda < 0)$.
Then we have $\rho(\lambda) = e^{-i k(\lambda) a}$ for all $\lambda\in\mathbb C$ such that $\mathfrak D(\lambda)\notin [-2,2]$.
Note that this extended quasimomentum function is not continuous at the real axis;
nevertheless, its imaginary part is continuous as
$\Imag k(\lambda) = -(-\Imag k(\overline\lambda)) = \Imag k(\overline\lambda)$.

We now express the resolvent operator $(H_0 - \lambda)^{-1}$ in terms of a
fundamental system of Floquet solutions.
Let $\lambda\in\mathbb C$ such that $\mathfrak D(\lambda)\notin [-2, 2]$.
Then the Floquet solutions
\begin{align}
 u_+(x,\lambda) &= \Phi(x,\lambda)\,v_+(\lambda) = \rho(\lambda)^{x/a}\,\varphi_+(x, \lambda),
\nonumber
\\
 u_-(x,\lambda) &= \Phi(x, \lambda)\,v_-(\lambda) = \rho(\lambda)^{-x/a}\,\varphi_-(x, \lambda)
\label{eq:Flop}
\end{align}
with $a$-periodic functions $\varphi_\pm(\cdot, \lambda)$ are linearly independent and hence form a fundamental
system of the Dirac equation.
As $u_\pm(0, \lambda) = \varphi_\pm(0, \lambda) = v_\pm(\lambda)$, its Wronskian is
$W(\lambda) = \det(v_+(\lambda), v_-(\lambda))$.

\begin{thm}\label{thm:resolvent}
Let $\lambda\in\varrho(H_0)$. Then
\begin{equation*}
 ((H_0 - \lambda)^{-1} f)(x) = \int_{\mathbb R} G(x, t, \lambda) \, f(t) \, d t
 \qquad (x\in\mathbb R; f\in L^2(\mathbb R)^2)
\end{equation*}
with (matrix-valued) Green's function
\begin{equation*}
 G(x, t, \lambda) = -\frac {e^{i k(\lambda)\, |t - x|}}{\det(v_+(\lambda), v_-(\lambda))} \left\{\begin{matrix}
 \varphi_+(x, \lambda)\,\varphi_-(t, \lambda)^T & \hbox{\it if\/}\ t > x \\
 \varphi_-(x, \lambda)\,\varphi_+(t, \lambda)^T & \hbox{\it if\/}\ t < x
 \end{matrix} \right.
 \quad (x, t \in \mathbb R).
\end{equation*}
For all $x, t \in \mathbb R$, $x \neq t$, the Frobenius norm of the matrix
$G(x, t, \lambda)$ is
\begin{equation*}
\|G(x, t, \lambda)\|_F = \frac{e^{-\Imag k(\lambda)\,|t - x|}}{|\det(v_+(\lambda), v_-(\lambda))|} \left\{\begin{matrix}
 |\varphi_+(x, \lambda)|\,|\varphi_-(t, \lambda)| & \hbox{\it if\/}\ t > x,
\\
 |\varphi_-(x, \lambda)|\,|\varphi_+(t, \lambda)| & \hbox{\it if\/}\ t < x.
 \end{matrix}\right.
\end{equation*}
Moreover, for any $r \in (1, 2]$ and conjugate exponent $r' = 1/(1 - \frac 1 r) \ge 2$, the integral operator
$R_r(\lambda) : L^r(\mathbb R)^2 \rightarrow L^{r'}(\mathbb R)^2$,
\begin{equation*}
(R_r(\lambda) f)(x) = \int_{\mathbb R} G(x, t, \lambda)\,f(t)\,d t
\qquad (x \in \mathbb R; f \in L^r(\mathbb R)^2)
\end{equation*}
is a bounded linear operator with operator norm
$\|R_r(\lambda)\| \le C(\lambda)\left(\frac{4}{r'\,\Imag k(\lambda)}\right)^{\frac 2{r'}}$,
where
\begin{equation}
C(\lambda) := \frac{\|\varphi_+(\cdot, \lambda)\|_\infty\,\|\varphi_-(\cdot, \lambda)\|_\infty}{|\det(v_+(\lambda), v_-(\lambda))|}.
\label{eq:Cdef}
\end{equation}
\end{thm}
\noindent
{\it Remarks.\/}
1. The Green's function $G$ is in fact independent of the choice of the
eigenvectors $v_\pm(\lambda)$.

2.
In the absence of a periodic background potential $q$, an operator norm bound
for $R_r(\lambda)$ was obtained in \cite[Theorem 3.1]{JCC}.
\begin{proof}
Let $f\in L^2(\mathbb R)^2$; then solving the inhomogeneous Dirac equation
\begin{equation*}
-i \sigma_2 u'(x) + (m \sigma_3 + q(x) - \lambda)\,u(x) = f(x) \qquad (x\in\mathbb R)
\end{equation*}
by the variation of constants method on the basis of the fundamental system
$(u_+(\cdot,\lambda), u_-(\cdot,\lambda))$ gives
\begin{equation*}
 u(x) = \int_{\mathbb R} G(x, t, \lambda)\,f(t)\,d t \qquad (x\in\mathbb R).
\end{equation*}
For $x \neq t$, the Frobenius norm of the matrix $G(x, t, \lambda)$ is
\begin{align*}
\|G(x, t, \lambda)\|_F &= \sqrt{\tr (G(x, t, \lambda)^*\ G(x, t, \lambda))}
\\
&= \frac{|e^{i k(\lambda)\,|t-x|}|}{|\det(v_+(\lambda), v_-(\lambda))|}
  \sqrt{\tr(\overline{\varphi_\mp(t, \lambda)}\,\varphi_\pm(x, \lambda)^*\,\varphi_\pm(x, \lambda)\,\varphi_\mp(t, \lambda)^T)}
\\
&= \frac{e^{-\Imag k(\lambda)\,|t-x|}}{|\det(v_+(\lambda), v_-(\lambda))|}
  \sqrt{\tr(\varphi_\pm(x, \lambda)^*\,\varphi_\pm(x, \lambda)\,\varphi_\mp(t, \lambda)^T\,\overline{\varphi_\mp(t, \lambda)})}
\\
&= \frac{e^{-\Imag k(\lambda)\,|t-x|}}{|\det(v_+(\lambda), v_-(\lambda))|}
  \sqrt{|\varphi_\pm(x, \lambda)|^2\,|\varphi_\mp(t, \lambda)|^2}
\end{align*}
with the sign in the index depending on whether $t > x$ or $t < x$.
Setting
$\|\varphi_\pm(\cdot,\lambda)\|_\infty := \sup\limits_{x\in\mathbb R} |\varphi_\pm(\cdot,\lambda)|$, we can estimate the operator norm
\begin{equation}
\|G(x, t, \lambda)\| \le \|G(x, t, \lambda)\|_F
\le C(\lambda)  \, e^{-\Imag k(\lambda) |t - x|}
\label{eq:Frobest}
\end{equation}
$(x, t \in \mathbb R, t \neq x)$
with
$C(\lambda)$ defined in equation (\ref{eq:Cdef}).
Now let $f \in L^r(\mathbb R)$; then
\begin{align*}
\|R_r(\lambda) f\|_{r'} &= \left(\int_{\mathbb R} \left|\int_{\mathbb R} G(x, t, \lambda)\,f(t)\,d t\right|^{r'} d x \right)^{\frac 1{r'}}
\\
&\le \left(\int_{\mathbb R} \left(\int_{\mathbb R} \|G(x, t, \lambda)\|\,|f(t)|\,d t\right)^{r'} d x \right)^{\frac 1{r'}}
\\
&\le C(\lambda) \left(\int_{\mathbb R} \left(\int_{\mathbb R} e^{-\Imag k(\lambda)\,|t-x|}\,|f(t)|\,d t\right)^{r'} d x\right)^{\frac 1{r'}}
\\
&\le C(\lambda) 
 \left(\int_{\mathbb R} e^{-\Imag k(\lambda)\,|s|\,\frac{r'}2}\,d s\right)^{\frac 2{r'}}
 \left(\int_{\mathbb R} |f(x)|^r\,d x\right)^{\frac 1r}
\\
&= C(\lambda)\left(\frac 4{r'\,\Imag k(\lambda)}\right)^{\frac 2{r'}} \|f\|_r
\end{align*}
by Young's inequality, noting that $\frac 1r + \frac 2{r'} = \frac 1{r'} + 1$.
This shows that the integral operator $R_r(\lambda)$ (and in particular the
resolvent operator $(H_0-\lambda)^{-1} = R_2(\lambda)$) is well-defined and
bounded, with the stated operator norm estimate.
\end{proof}
\section{Eigenvalue exclusion}\label{sec:exclusion}
We now consider the Dirac operator with an additional non-periodic perturbation,
$H := H_0 + V$, where $V$ is the operator of multiplication with the
matrix-valued function $V : \mathbb R \rightarrow \mathbb C^{2\times 2}$.
We assume that $V$ is bounded and, for some $p \ge 1$, $V \in L^p(\mathbb R)^{2\times 2}$, which means that the norm
(cf.\ \cite{CS})
\begin{equation}
 \|V\|_p := \left(\int_{\mathbb R} \|V(x)\|^p\,d x \right)^{\frac 1 p}
\label{eq:Vnorm}
\end{equation}
is finite.
Here $\|V(x)\|$ is the operator norm of the matrix $V(x)$, $x\in \mathbb R$.
This is different from the operator norm $\|V\|$ of the multiplication operator
$V$ in $L^2(\mathbb R)^2$.

For each $x\in\mathbb R$, we use the polar decomposition of $V(x)$,
\begin{equation}
 V(x) = B(x) A(x), \qquad B(x) = U(x)\,|V(x)|^{\frac 12}, \qquad A(x) = |V(x)|^{\frac 12},
\label{eq:podeco}
\end{equation}
where $|V(x)| = (V(x)^*\ V(x))^{\frac 12}$ and $U(x)$ is a partial isometry of
$\mathbb C^2$, cf.\ \cite[Theorem VI.10]{RS1};
then
\begin{equation}
\|A(x)\| = \sqrt{\|V(x)\|}, \quad\|B(x)\| \le \sqrt{\|V(x)\|}
\qquad (x \in \mathbb R).
\label{eq:ABnorm}
\end{equation}
Thus we have matrix-valued functions
$A, B \in L^{2p}(\mathbb R)^{2 \times 2}$ that
give rise to bounded multiplication operators $A, B$ on $L^2(\mathbb R)^2$.

As we don't assume that $V$ is symmetric, the operator $H$ is not self-adjoint
in general; however, as a sum of a closed (self-adjoint) operator and a
bounded operator, it is closed (cf.\ \cite[Theorem 5.5]{W}).
Moreover, we have the following statement about its essential spectrum, using any of the 5 usual definitions (cf.\ \cite[Section I.4]{EE}), e.g.\ the third,
\begin{equation*}
\sigma_e(H) := \{\lambda\in\mathbb C \mid H - \lambda \ \hbox{\rm is not a
Fredholm operator}\}.
\end{equation*}
\begin{thm}\label{thm:compres}
$\sigma_e(H) = \sigma_e(H_0) = \{\lambda\in\mathbb R \mid \mathfrak D(\lambda) \in [-2, 2]\}$.
The spectrum of $H$ outside $\sigma_e(H)$ only consists of isolated eigenvalues
of finite multiplicity.
\end{thm}
In the proof of this theorem, we use the relative compactness of the multiplication operator $A$ with respect to $H_0$. We give a full proof of this statement (which holds for any $A\in L^{2p}(\mathbb R)^{2\times 2}$), as it
does not seem to be easily available in the
literature; note that we only assume that the periodic potential $q$ is locally integrable,
so the results of e.g.\ \cite[Theorem 4.1]{Simon} or \cite[Theorem 4.1]{JCC} are not directly applicable.
We remark that in the case $p=1$ the relative compactness can be shown more
easily by proving that $A\, (H_0 - \lambda)^{-1}$, an integral operator with
kernel $A(x)\,G(x, t, \lambda)$, is a Hilbert-Schmidt operator, using the
Frobenius norm estimate (\ref{eq:Frobest}).
\begin{lem}\label{lem:compres}
Let $\lambda\in\varrho(H_0)$. Then the operator $A\,(H_0 - \lambda)^{-1}$ is
compact.
\end{lem}
\begin{proof}
(a)
We first show that the statement is true for $A \in C_0^\infty(\mathbb R)^{2\times 2}$.
Let ${a < b}$ be such that $\mathop{\rm supp} A \subset [a, b]$.
Let $(u_n)_{n\in\mathbb N}$ be a bounded sequence in $L^2(\mathbb R)^2$,
$\|u_n\|_2 \le K$ $(n\in\mathbb N)$,
and set $y_n := A\,(H_0 - \lambda)\,u_n$ $(n\in\mathbb N)$.
Then, for all $x \in [a, b]$ and $n\in\mathbb N$,
we have by Theorem \ref{thm:resolvent}
\begin{align*}
|y_n(x)| &\le \|A(x)\|\left|\int_{\mathbb R} G(x, t, \lambda)\,u_n(t)\,d t \right|
\le \|A(x)\| \left(\int_{\mathbb R} \|G(x, t, \lambda)\|^2\,d t\right)^{\frac 12} \|u_n\|_2
\\
&\le \sup_{z\in\mathbb R} \|A(z)\|\left(\int_{\mathbb R} e^{-2 \Imag k(\lambda)\,|t|}\,d t\right)^{\frac 12}\,C(\lambda)\,K
< \infty
\end{align*}
and $y_n(x) = 0$ for all $x\in\mathbb R\setminus [a,b]$, so the sequence of
functions $(y_n)_{n\in\mathbb N}$ is uniformly bounded.
Also, for $a\le x < z\le b$ we find, using the estimate (\ref{eq:Frobest}),
\begin{align*}
&|y_n(x) - y_n(z)| = \left| A(x)\int_{\mathbb R} G(x, t, \lambda)\,u_n(t)\,d t - A(z)\int_{\mathbb R} G(z, t, \lambda)\,u_n(t)\,d t \right|
\\
&\le \frac 1W \Bigg(
\int_{-\infty}^x \left|\left(A(x)\,e^{i k(\lambda)\,(x-t)}\varphi_-(x) - A(z)\,e^{i k(\lambda)\,(z-t)}\varphi_-(z)\right)\varphi_+(t)^T\, u_n(t) \right| d t
\\
&\ +
\int_x^z \left|\left(A(x)e^{i k(\lambda)(t-x)}\varphi_+(x)\varphi_-(t)^T - A(z)e^{i k(\lambda)(z-t)}\varphi_-(z)\varphi_+(t)^T\right) u_n(t) \right| d t
\\
&\ +
\int_z^\infty \left|\left(A(x)\,e^{i k(\lambda)\,(t-x)}\,\varphi_+(x) - A(z)\,e^{i k(\lambda)\,(t-z)}\,\varphi_+(z)\right)\varphi_-(t)^T\, u_n(t) \right| d t\Bigg),
\end{align*}
where we abbreviated $W := \det(v_+(\lambda), v_-(\lambda))$.
Here the first integral is less than or equal to
\begin{equation*}
\left|e^{i k(\lambda) x} A(x)\,\varphi_-(x) - e^{i k(\lambda) z} A(z)\,\varphi_-(z)\right| \left(\int_{-\infty}^b e^{2\Imag k(\lambda) t} |\varphi_+(t)|^2 \,dt\right)^{\frac 12}\,K,
\end{equation*}
and as $A$ and $\varphi_-$ are continuous and hence uniformly continuous on $[a, b]$, this integral tends to 0 as $|x-z| \rightarrow 0$ uniformly on $[a, b]$ and in $n\in\mathbb N$.
Analogous reasoning applies to the third integral.
The second integral can be estimated by
\begin{align*}
&\left(\int_x^z \|A(x) e^{i k(\lambda)\,|t-x|} \varphi_+(x) \varphi_-(t)^T - A(z) e^{i k(\lambda)\,|t-z|} \varphi_-(z) \varphi_+(t)^T\|^2 dt\right)^{\frac 12}
\\
&\qquad \times
\left(\int_x^z |u_n(t)|^2 dt\right)^{\frac 12}
\le 2 \sup_{t\in[a,b]} \|A(t)\| \sup_{t\in[a,b]} |\varphi_+(t)| \sup_{t\in[a,b]} |\varphi_-(t)|\,K\,\sqrt{z-x},
\end{align*}
which tends to 0 as $|x-z|\rightarrow 0$ uniformly on $[a,b]$ and in $n\in\mathbb N$.
Consequently, the sequence of functions $(y_n)_{n\in\mathbb N}$ is also equicontinuous. By the Arzel\`a-Ascoli Theorem, it has a subsequence that is uniformly
convergent and hence, in view of the compact support, also converges in $L^2(\mathbb R)^2$.

(b)
Now let $A\in L^{2p}(\mathbb R)^{2\times 2}$.
Let $u, v\in L^2(\mathbb R)^2$, $\|u\|_2 = \|v\|_2 = 1$. Then
\begin{align*}
|(A\,&(H_0 - \lambda)^{-1} u, v)| = \left|\int_{\mathbb R} A(x) \int_{\mathbb R} G(x, t, \lambda)\,u(t)\,dt\,\overline{v(x)}\,dx\right|
\\
&\le C(\lambda) \int_{\mathbb R} |u(t)| \left(\int_{\mathbb R} e^{-\Imag k(\lambda)\,|t-x|}\,\|A(x)\|\,|v(x)|\,dx\right)\,dt
\\
&= C(\lambda) \int_{\mathbb R} \overline{F(|u|)}\,F\left(e^{-\Imag k(\lambda)\,|\cdot|} * (\|A(\cdot)\|\,|v|)\right)
\\
&= C(\lambda) \int_{\mathbb R} \overline{F(|u|)} \sqrt{2\pi}\,F(e^{-\Imag k(\lambda)\,|\cdot|})\,F(\|A(\cdot)\|\,|v|)
\\
&= 2\Imag k(\lambda)\,C(\lambda) \int_{\mathbb R} F(\|A(\cdot)\|\,|v|)(\xi)\,\frac{\overline{F(|u|)(\xi)}}{\xi^2 + (\Imag k(\lambda))^2}\,d\xi
\\
&\le 2\Imag k(\lambda)\,C(\lambda) \left(\int_{\mathbb R} \left|F(\|A(\cdot)\|\,|v|)\right|^{r'}\right)^{\frac 1{r'}}
\left(\int_{\mathbb R} \left|\frac{F(|u|)(\xi)}{\xi^2 + (\Imag k(\lambda))^2}\right|^r d\xi\right)^{\frac 1r},
\end{align*}
where we used the Plancherel identity for the Fourier transform $F$ and
then H\"older's inequality with exponent $r := \frac{2p}{p+1} \in [1, 2)$ and conjugate
exponent $r'$.
(In the case $p=1$, where $r = 1$, the above and the following estimates hold with
$\left(\int_{\mathbb R} |F(\|A(\cdot)\|\,|v|)|^{r'} \right)^{1/r'}$
replaced with
$\sup_{x\in\mathbb R} |F(\|A(\cdot)\|\,|v|)(x)|$.)
By the Hausdorff-Young inequality,
\begin{align*}
\left(\int_{\mathbb R} |F(\|A(\cdot)\|\,|v|)|^{r'}\right)^{\frac 1{r'}}
\!\!&\le \sqrt{2\pi}^{1 - \frac 2r} \left(\int_{\mathbb R} \|A(\cdot)\|^r\,|v|^r\right)^{\frac 1r}
\\
&= \frac 1{\sqrt{2\pi}^{\frac 1p}}\left(\int_{\mathbb R} \left(\|A(\cdot)\|^{2p}\right)^{\frac 1q} \left(|v|^2\right)^{\frac 1{q'}}\right)^{\frac 1r}
\\
&\le \frac 1{\sqrt{2\pi}^{\frac 1p}} \left(\int_{\mathbb R} \|A(\cdot)\|^{2p}\right)^{\frac 1{rq}} \left(\int_{\mathbb R} |v|^2\right)^{\frac 1{rq'}}
\!\!= \frac 1{\sqrt{2\pi}^{\frac 1p}} \|A\|_{2p},
\end{align*}
using H\"older's inequality with exponents $q := p+1 = \frac{2p}r$ and
$q' = \frac{p+1}p = \frac 2r$.
The same H\"older inequality gives
\begin{equation*}
\left(\int_{\mathbb R} \left|\frac{F(|u|)(\xi)}{\xi^2 + (\Imag k(\lambda))^2}\right|^r d\xi\right)^{\frac 1r}
\le \left(\int_{\mathbb R} (\xi^2 + (\Imag k(\lambda))^2)^{-2p}d\xi\right)^{\frac 1{2p}} \|F(|u|)\|_2.
\end{equation*}
As $\|F(|u|)\|_2 = \|u\|_2 = 1$, taking the supremum over $u, v$ gives the bound for the operator norm
\begin{equation*}
\|A\,(H_0 - \lambda)^{-1}\| \le 2\Imag k(\lambda)\,C(\lambda) \left(\frac 1{2\pi}\int_{\mathbb R} (\xi^2 + (\Imag k(\lambda))^2)^{-2p}\,d\xi\right)^{\frac 1{2p}} \|A\|_{2p}.
\end{equation*}
As $C_0^\infty(\mathbb R)$ is dense in $L^{2p}(\mathbb R)$, there is a sequence
$(A_n)_{n\in\mathbb N}$ in $C_0^\infty(\mathbb R)^{2\times 2}$ that converges
to $A$ in $\|\cdot\|_{2p}$; by the above estimate, $A_n\,(H_0 - \lambda)^{-1}$
converges to $A\,(H_0 - \lambda)^{-1}$ in operator norm and the statement of the lemma follows from (a) and the fact that the space of compact operators is closed in the operator norm.
\end{proof}
We are now ready to prove Theorem \ref{thm:compres}.
\begin{proof}[Proof of Theorem \ref{thm:compres}]
The resolvent set of $H$, $\varrho(H)$, contains points in the upper and the
lower complex half-planes, as $\lambda \in \varrho(H)$ if $|\Imag\lambda| > \|V\|$.
By the resolvent identity, we find for $\lambda\in\varrho(H)\cap\varrho(H_0)$
\begin{equation*}
(H_0 - \lambda)^{-1} - (H - \lambda)^{-1} = (H - \lambda)^{-1}\,B\,A\,(H_0 - \lambda)^{-1}.
\end{equation*}
As $(H - \lambda)^{-1}$ and $B$ are bounded operators, this resolvent difference is compact by Lemma \ref{lem:compres}.
We can now apply Theorem IX.2.4 of \cite{EE} to conclude the equality of the
essential spectra (all five types) of $H$ and of $H_0$.

The complement of the essential spectrum of $H$, $\mathbb C \setminus \sigma_e(H)$, is open and either connected (if $H_0$ has at least one spectral gap) or
has the upper and lower complex half-planes as connected components (if
$\sigma(H_0) = \mathbb R$ --- this happens if $m = 0$, see
\cite[Proposition 1]{KMS1}).
In either case, each component of the complement of $\sigma_e(H)$ contains
points of the resolvent set $\varrho(H)$, and we can therefore apply
Theorem XVII.2.1 of \cite{GGK} to conclude that the spectrum of $H$ outside
$\sigma_e(H)$ only consists of isolated eigenvalues of finite multiplicity.
\end{proof}

In the statement of the eigenvalue exclusion theorem, we use the function
$\Gamma : D(\Gamma) \rightarrow (0, 1]$,
$D(\Gamma) = \{A \in \mathbb C^{2\times 2} \mid A\ \hbox{\rm has two distinct
eigenvalues}\}$,
\begin{equation}
\ \Gamma(A) = \frac{|\det (v_+, v_-)|}{|v_+|\,|v_-|},
\label{eq:Gammadef}
\end{equation}
where $v_\pm \in \mathbb C^2\setminus \{0\}$ are eigenvectors of $A$ for the
two different eigenvalues.
As the eigenvectors are uniquely determined up to a complex factor, $\Gamma(A)$
does not depend on the choice of eigenvectors and is therefore well-defined.
The domain $D(\Gamma)$ is an open subset of $\mathbb C^{2\times 2}$ and
$\Gamma$ is continuous. However, $\Gamma$ cannot be continuously extended to
all of $\mathbb C^{2\times 2}$; for example,
\begin{equation*}
\lim_{\varepsilon\rightarrow 0} \Gamma\begin{pmatrix} 1 & 0 \\ 0 & 1 + \varepsilon \end{pmatrix} = 1 \neq
 0 = \lim_{\varepsilon\rightarrow 0} \Gamma\begin{pmatrix} 1 & \varepsilon \\ \varepsilon^2 & 1 \end{pmatrix},
\end{equation*}
so $\Gamma$ has no continuous extension at the unit matrix.
For the monodromy matrix $M$ of equation (\ref{eq:pD}), we have the following
statement.
\begin{lem}\label{lem:monogam}
For all $\lambda\in\mathbb C$, $M(\lambda)\in D(\Gamma)$ if and only if
$\mathfrak D(\lambda) \notin \{-2, 2\}$.
\end{lem}
We can now state the main eigenvalue exclusion theorem.
\begin{thm}\label{thm:eet}
Let
$p \ge 1$ and let
$V \in L^p(\mathbb R)^{2\times 2} \cap L^\infty(\mathbb R)^{2\times 2}$.
Then
$\lambda \in \mathbb C \setminus \sigma_e(H)$
is not an eigenvalue of $H$ if
\begin{align*}
 \|V\|_1 &< \Gamma(M(\lambda))\,\gamma_+(\lambda)\,\gamma_-(\lambda)
 \qquad (\hbox{\it if $p=1$}),
\\
 \|V\|_p &< \Gamma(M(\lambda))\,\gamma_+(\lambda)\,\gamma_-(\lambda)
 \,(\Imag k(\lambda))^{\frac{p-1}p} \left(\frac p{2(p-1)}\right)^{\frac {p-1}p}
 \qquad (\hbox{\it if $p > 1$}),
\end{align*}
where
\begin{equation}
\gamma_\pm(\lambda) = \frac{|\varphi_\pm(0,\lambda)|}{\sup_{x\in[0,a]} |\varphi_\pm(x,\lambda)|}
\label{eq:gammadef}
\end{equation}
and $\varphi_\pm$ are the periodic functions in equation $(\ref{eq:Flop})$.
\end{thm}
\noindent
{\it Remark.\/} The additional factor that appears on the right-hand side of
the inequality in Theorem \ref{thm:eet} for $p > 1$ tends to 1 as $p \rightarrow 1$, so the exclusion criterion is formally continuous in $p$.
\begin{proof}
By the Birman-Schwinger Principle (see e.g.\ \cite[Theorem B.2]{BKN}), $\lambda$ is an eigenvalue of $H_0 + V$
if and only if $-1$ is an eigenvalue of $A\,(H_0 - \lambda)^{-1}\,B$,
where $A$, $B$ are as in equation (\ref{eq:podeco}).

{\it Case 1:\/} $p = 1$.
For $u, v \in L^2(\mathbb R)^2$, we obtain from Theorem 1 and the estimate (\ref{eq:Frobest}),
noting that $e^{-\Imag k(\lambda)\,|t - x|} \le 1$, 
\begin{align*}
\left|(A(H_0-\lambda)^{-1} B u, v)\right| &= \left|\int_{-\infty}^\infty \left( A(x) \int_{-\infty}^\infty G(x, y, \lambda) B(y) u(y)\,d y\right)^T \overline{v(x)}\,d x \right|
\\
&\le \int_{-\infty}^\infty \|A(x)\|\,\|G(x, y, \lambda)\|\,\|B(y)\|\,|u(y)|\,|v(x)|\,d y\,d x
\\
&\le C(\lambda)\left(\int_{-\infty}^\infty \|A(x)\|\,|v(x)|\,d x\right)\left(\int_{-\infty}^\infty \|B(y)\|\,|u(y)|\,d y\right)
\\
&\le 
 C(\lambda)\,\|V\|_1\,\|v\|_2\,\|u\|_2,
\end{align*}
where we used H\"older's inequality and the estimate (\ref{eq:ABnorm}) in the
last step.
Setting $v := A (H_0-\lambda)^{-1} B u$ and taking the supremum over $u$, we hence find
the estimate for the operator norm of the Birman-Schwinger kernel
\begin{equation*}
 \|A (H_0-\lambda)^{-1} B\| \le C(\lambda)\,\|V\|_1.
\end{equation*}

{\it Case 2:\/} $p > 1$.
Let $r := \frac{2p}{p+1} \in (1, 2)$ with conjugate exponent $r' = \frac{2p}{p-1} > 2$. We now associate the matrix-valued functions $A$ and $B$ with
multiplication operators
$A_{r',2} : L^{r'}(\mathbb R)^2 \rightarrow L^2(\mathbb R)^2$,
$B_{2, r} : L^2(\mathbb R)^2 \rightarrow L^r(\mathbb R)^2$
and write the Birman-Schwinger kernel as
$A (H_0 - \lambda)^{-1} B = A_{r',2} R_r(\lambda) B_{2, r}$,
where the operator $R_r(\lambda) : L^r(\mathbb R)^2 \rightarrow L^{r'}(\mathbb R)^2$ is defined as in Theorem \ref{thm:resolvent}.
For $u\in L^{r'}(\mathbb R)^2$, we find
\begin{align*}
\|A_{r',2} u\|_2 &= \left(\int_{\mathbb R} |A(x)\,u(x)|^2\,dx\right)^{\frac 12}
\le \left(\int_{\mathbb R} \|A(x)\|^2 \left(|u(x)|^{r'}\right)^{\frac 2{r'}} dx\right)^{\frac 12}
\\
&\le \left(\int_{\mathbb R} \|V(x)\|^{\frac{r'}{r'-2}} dx\right)^{\frac{r'-2}{2r'}} \|u\|_{r'},
\end{align*}
using H\"older's inequality with exponents $\frac {r'}2$ and $\frac {r'}{r'-2}$.
As $\frac {r'}{r'-2} = p$, we obtain the operator norm estimate
$\|A_{r',2}\| \le \|V\|_p^{\frac 12}$.
Similarly, we find for $u \in L^2(\mathbb R)^2$
\begin{align*}
\|B_{2,r} u\|_r &= \left(\int_{\mathbb R} |B(x)\,u(x)|^r\,dx\right)^{\frac 1r}
\le \left(\int_{\mathbb R} \|B(x)\|^r \left(|u(x)|^2\right)^{\frac r2} dx\right)^{\frac 1r}
\\
&\le \left(\int_{\mathbb R} \|V(x)\|^{\frac r{2-r}} dx\right)^{\frac {2-r}{2r}} \|u\|_2,
\end{align*}
using H\"older's inequality with exponents $\frac 2r$ and $\frac 2{2-r}$,
and hence, as $\frac r{2-r} = p$, the operator norm estimate
$\|B_{2,r}\| \le \|V\|_p^{\frac 12}$.
In conjunction with Theorem \ref{thm:resolvent}, we obtain
\begin{align*}
\|A (H_0 - \lambda)^{-1} B\| &= \|A_{r',2} R_r(\lambda) B_{2,r}\|
\le \|A_{r',2}\|\,\|R_r(\lambda)\|\,\|B_{2,r}\|
\\
&\le C(\lambda) \left(\frac 2{\Imag k(\lambda)}\,\frac {p-1}p\right)^{\frac {p-1}p} \|V\|_p,
\end{align*}
since $\frac 2{r'} = \frac{p-1}p$.

Now if $\lambda$ is an eigenvalue of $H$, then $-1$ is an eigenvalue of $A (H_0-\lambda) B$ and therefore $\|A (H_0-\lambda)^{-1} B\| \ge 1$;
noting that $\frac 1{C(\lambda)} = \Gamma(M(\lambda))\,\gamma_+(\lambda)\,\gamma_-(\lambda)$ since $\varphi_\pm(0,\lambda) = v_\pm(\lambda)$, we obtain the eigenvalue exclusion criteria in the theorem by contraposition.
\end{proof}
\section{Behaviour near the essential spectrum}\label{sec:near}
In this section we study the behaviour of the right-hand side of the inequalities
in Theorem \ref{thm:eet}, in particular as $\lambda$ approaches the essential spectrum $\sigma_e(H)$.
We begin by finding a positive lower bound for the factors $\gamma_\pm(\lambda)$ defined in equation (\ref{eq:gammadef}).
\begin{thm}\label{thm:gammabound}
(a)
Let $\lambda\in\mathbb C$ such that $\mathfrak D(\lambda) \notin [-2, 2]$.
Then
\begin{equation*}
e^{-a \left(\Imag k(\lambda) + \sqrt{m^2 + (\Imag \lambda)^2}\right)} \le
\gamma_\pm(\lambda) \le 1.
\end{equation*}

\noindent
(b)
Let $\lambda_0 \in \mathbb R$ be such that $\mathfrak D(\lambda_0) \in [-2, 2]$. Then
\begin{equation*}
e^{-a m} \le \liminf_{\lambda\rightarrow\lambda_0} \gamma_\pm(\lambda) \le 1.
\end{equation*}
\end{thm}
\begin{proof}
(a)
The upper bound is immediate from the definition of $\gamma_\pm$. For the lower
bound, we note that $|\varphi_\pm|^2$ satisfies the differential equation
\begin{equation}
\frac d{d x} |\varphi_\pm(x,\lambda)|^2 = \varphi_\pm(x,\lambda)^T \mathcal B_\pm(\lambda) \overline{\varphi_\pm(x,\lambda)}
\qquad (x\in\mathbb R),
\label{eq:phinormde}
\end{equation}
with
\begin{equation*}
 \mathcal B_\pm(\lambda) = 2 \begin{pmatrix} \mp \Imag k(\lambda) & m - i \Imag \lambda \\ m + i \Imag \lambda & \mp \Imag k(\lambda) \end{pmatrix}.
\end{equation*}
Indeed, the Floquet solutions $u_\pm(\cdot,\lambda)$ are solutions of the differential
equation (\ref{eq:Dirac}), which can be rewritten in the form
\begin{equation*}
 u'(x,\lambda) = \begin{pmatrix} 0 & m - q(x) +\lambda \\ m + q(x) - \lambda & 0 \end{pmatrix} u(x,\lambda)
\qquad (x \in \mathbb R);
\end{equation*}
so by differentiation of
$\varphi_\pm(x,\lambda) = u_\pm(x,\lambda)\,e^{\pm i k(\lambda) x}$
$(x \in \mathbb R)$
we find that
$\varphi_\pm'(x,\lambda) = B_\pm(x,\lambda)\,\varphi_\pm(x,\lambda)$, where
\begin{equation*}
B_\pm(x,\lambda) = \begin{pmatrix} \pm i\,k(\lambda) & m - q(x) + \lambda \\ m + q(x) - \lambda & \pm i\,k(\lambda)\end{pmatrix} \qquad (x\in\mathbb R).
\end{equation*}
Hence
\begin{align*}
\frac d{d x}\,|\varphi_\pm(x,\lambda)|^2 &= \varphi_\pm(x,\lambda)^T\,\overline{\varphi_\pm'(x,\lambda)} + \varphi_\pm'(x,\lambda)^T\,\overline{\varphi_\pm(x,\lambda)}
\\
&= \varphi_\pm(x,\lambda)^T \left(\overline{B_\pm(x,\lambda)} + B_\pm(x,\lambda)^T\right) \overline{\varphi(x,\lambda)}
\end{align*}
and equation (\ref{eq:phinormde}) follows noting that
$\mathcal B_\pm(\lambda) = \overline{B_\pm(x,\lambda)} + B_\pm(x,\lambda)^T$
does not depend on $x$.
From (\ref{eq:phinormde}),
\begin{align*}
 \frac d{d x} |\varphi_\pm(x,\lambda)|^2 &\le \left| \frac d{d x} |\varphi_\pm(x,\lambda)|^2 \right|
= |\varphi_\pm(x,\lambda)^T \mathcal B_\pm(\lambda) \overline{\varphi_\pm(x, \lambda)}|
\\
&\le \|\mathcal B_\pm(\lambda)\|\,|\varphi_\pm(x,\lambda)|^2.
\end{align*}
To find the operator norm of the symmetric matrix $\mathcal B_\pm(\lambda)$,
we calculate its eigenvalues
$\mp 2 \Imag k(\lambda) + 2 \sqrt{m^2 + (\Imag \lambda)^2}$ and
$\mp 2 \Imag k(\lambda) - 2 \sqrt{m^2 + (\Imag \lambda)^2}$, and hence the
spectral radius
\begin{equation*}
 \|\mathcal B_\pm(\lambda)\| = 2 \Imag k(\lambda) + 2 \sqrt{m^2 + (\Imag \lambda)^2}.
\end{equation*}
Hence the above differential inequality gives
\begin{equation*}
|\varphi_\pm(x,\lambda)|^2 \le |\varphi_\pm(0,\lambda)|^2\,e^{(2 \Imag k(\lambda) + 2 \sqrt{m^2 + (\Imag \lambda)^2})\,x}
\qquad (x \in [0, a])
\end{equation*}
and so the lower bound in the Theorem.

\noindent
(b)
By part (a), we have
\begin{equation*}
\gamma_\pm(\lambda) - e^{-a\left(\Imag k(\lambda) + \sqrt{m^2 + (\Imag\lambda)^2}\right)} \ge 0
\end{equation*}
for all $\lambda\in\mathbb C$ such that $\mathfrak D(\lambda) \notin \{-2, 2\}$,
so using equation (\ref{eq:Imklim}), we find
\begin{align*}
 0 &\le \liminf_{\lambda\rightarrow\lambda_0} \left(\gamma_\pm(\lambda) - e^{-a\left(\Imag k(\lambda) + \sqrt{m^2 + (\Imag\lambda)^2} \right)}\right)
\\
&\le \liminf_{\lambda\rightarrow\lambda_0} \gamma_\pm(\lambda) - \lim_{\lambda\rightarrow\lambda_0} e^{-a\left(\Imag k(\lambda) + \sqrt{m^2 + (\Imag\lambda)^2}\right)}
= \liminf_{\lambda\rightarrow\lambda_0} \gamma_\pm(\lambda) - e^{-am}.
\end{align*}
\end{proof}
We now consider the function $\Gamma(M(\lambda))$,
which, as a composition of a continous and an entire function, is continuous.
By Lemma \ref{lem:monogam} and the definition of $\Gamma$, we see that $\Gamma(M(\lambda)) > 0$ for
all $\lambda\in\mathbb C$ for which $\mathfrak D(\lambda) \notin \{-2, 2\}$.
However, at the points where $\mathfrak D(\lambda) \in \{-2, 2\}$, $\Gamma(M(\lambda))$ is not defined. These points are the real values of $\lambda$ where either $M(\lambda) = \pm \mathbb I$ --- then $\mathfrak D'(\lambda) = 0$ and
$\lambda$ is an interior point of a spectral band where two instability
intervals meet --- or $M(\lambda) \neq \pm \mathbb I$ 
has eigenvalue $\pm 1$ with algebraic multiplicity 2, but geometric multiplicity 1
--- then $\mathfrak D'(\lambda) \neq 0$ and $\lambda$ is an end-point of a spectral band
(cf.\ \cite[Theorem 1.6.1]{BES}).
In the following we investigate the limiting behaviour of $\Gamma(M(\lambda))$ at these exceptional points.
We can characterise the derivative of $M$ at such points as follows.
Let $\lambda\in\mathbb R$. Denoting the columns of the canonical fundamental system $\Phi$ of
equation (\ref{eq:pD}) by $u$ and $v$, we have
\begin{equation}
 M'(\lambda)
 = M(\lambda) \begin{pmatrix} I_1(\lambda) & I_2(\lambda) \\ -I_3(\lambda) & -I_1(\lambda) \end{pmatrix}
\label{eq:Mdash}
\end{equation}
where
\begin{align*}
 I_1(\lambda) &:= \int_0^a (u_1(x,\lambda) v_1(x,\lambda) + u_2(x,\lambda) v_2(x,\lambda))\,d x, \\
 I_2(\lambda) &:= \int_0^a (v_1(x,\lambda)^2 + v_2(x,\lambda)^2)\,d x, \\
 I_3(\lambda) &:= \int_0^a (u_1(x,\lambda)^2 + u_2(x,\lambda)^2)\,d x
\end{align*}
(see \cite[eq. (1.6.4), (1.6.6)]{BES}).
Note that $I_1(\lambda)$, $I_2(\lambda)$ and $I_3(\lambda)$ are complex in
general; however, for real spectral parameter they are real and have the
following property.
\begin{lem}\label{lem:Iineq}
For any $\lambda\in\mathbb R$,
$I_1(\lambda)^2 < I_2(\lambda)\,I_3(\lambda)$.
In particular, $I_2(\lambda), I_3(\lambda) \neq 0$.
\end{lem}
\begin{proof}
Since $\lambda\in\mathbb R$, $u$ and $v$ are $\mathbb R^2$-valued continuous functions.
The Cauchy-Schwarz inequality in $L^2(0, a)^2$ then gives
\begin{align*}
I_1(\lambda)^2 &= \left(\int_0^a u(x,\lambda)^T v(x,\lambda)\,d x\right)^2
\\
&\le \int_0^a |u(x,\lambda)|^2\,d x \int_0^a |v(x,\lambda)|^2\,d x
= I_2(\lambda)\,I_3(\lambda)
\end{align*}
with equality if and only if
$u$ and $v$ are linearly dependent functions; however, the latter is impossible
as $u(0,\lambda) = \begin{pmatrix} 1 \\ 0 \end{pmatrix}$ and $v(0,\lambda) = \begin{pmatrix} 0 \\ 1 \end{pmatrix}$.
\end{proof}
We can now find the limit of $\Gamma(M(\lambda))$ at a point $\lambda_0\in\mathbb R$ where
$\mathfrak D(\lambda_0) = \pm 2$, distinguishing the cases of
$M(\lambda_0) = \pm \mathbb I$ and of $M(\lambda_0) \neq \pm \mathbb I$.
Note that in the following theorems \ref{thm:MisI} and \ref{thm:MisnotI},
the limits $\lambda \rightarrow \lambda_0$ allow complex $\lambda$.
\begin{thm}\label{thm:MisI}
Let $\lambda_0\in\mathbb R$ be such that $M(\lambda_0) = s \mathbb I$, where
$s\in\{-1,1\}$. Then
\begin{equation*}
\lim_{\lambda\rightarrow\lambda_0} \Gamma(M(\lambda)) = 2\,\frac{\sqrt{I_2(\lambda_0)\,I_3(\lambda_0) - I_1(\lambda_0)^2}}{I_2(\lambda_0) + I_3(\lambda_0)}
> 0.
\end{equation*}
\end{thm}
\begin{proof}
Since $M$ is entire, we have by equation (\ref{eq:Mdash}) for $\lambda\in\mathbb C$,
abbreviating $I_j := I_j(\lambda_0)$, $j\in\{1,2,3\}$,
\begin{align*}
 M(\lambda) &= M(\lambda_0) + M'(\lambda_0)\,(\lambda - \lambda_0) + R(\lambda - \lambda_0) \\
&= s \mathbb I + s \mathbb I \begin{pmatrix} I_1 & I_2 \\ -I_3 & -I_1 \end{pmatrix} (\lambda - \lambda_0) + R(\lambda - \lambda_0) \\
&= s \mathbb I + s\,(\lambda - \lambda_0)\,N(\lambda - \lambda_0),
\end{align*}
where
\begin{equation*}
N(\Lambda) := \begin{pmatrix} I_1 & I_2 \\ -I_3 & -I_1 \end{pmatrix} + \frac{R(\Lambda)}\Lambda \qquad (\Lambda\in\mathbb C).
\end{equation*}
Here $R(\Lambda)/\Lambda$ is analytic with $\lim\limits_{\Lambda\rightarrow 0} R(\Lambda)/\Lambda = 0$.
Clearly, $w\in\mathbb C^2$ is an eigenvector of $N(\lambda - \lambda_0)$ for eigenvalue $\mu\in\mathbb C$ if and only if it is an eigenvector of $M(\lambda)$ for eigenvalue $s\,(1+ (\lambda - \lambda_0)\,\mu)$.
Therefore $\Gamma(M(\lambda)) = \Gamma(N(\lambda - \lambda_0))$ and we only
need to find $\lim\limits_{\Lambda\rightarrow 0} \Gamma(N(\Lambda))$.

Using Lemma \ref{lem:Iineq}, we see that the matrix
$N(0) = \begin{pmatrix} I_1 & I_2 \\ -I_3 & -I_1 \end{pmatrix}$ has distinct
purely imaginary eigenvalues $\pm i \sqrt{I_2 I_3 - I_1^2}$ with corresponding
eigenvectors
\begin{equation*}
w_+ = \begin{pmatrix} I_2 \\ -I_1 + i \sqrt{I_2 I_3 - I_1^2} \end{pmatrix},
\qquad
w_- = \begin{pmatrix} -I_1 + i \sqrt{I_2 I_3 - I_1^2} \\ I_3 \end{pmatrix}.
\end{equation*}
Hence by equation (\ref{eq:Gammadef})
\begin{equation*}
\Gamma(N(0))^2 = \frac{|2(I_2 I_3 - I_1^2 + i I_1 \sqrt{I_2 I_3 - i_1^2})|^2}{I_2\,(I_2 + I_3)\,I_3\,(I_2 + I_3)}
 = 4\,\frac{I_2 I_3 - I_1^2}{(I_2 + I_3)^2}.
\end{equation*}
By analytic perturbation theory (see \cite[Theorem II.1.8]{Kato}), the
eigenspaces of $N(\Lambda)$ converge to those spanned by $w_+$ and $w_-$ as
$\Lambda\rightarrow 0$, and we conclude that
\begin{equation*}
\lim_{\lambda\rightarrow\lambda_0} \Gamma(M(\lambda)) = \lim_{\Lambda\rightarrow 0} \Gamma(N(\Lambda))
= \Gamma(N(0)) = 2\,\frac{\sqrt{I_2 I_3 - I_1^2}}{I_2 + I_3},
\end{equation*}
which is positive by Lemma \ref{lem:Iineq}.
\end{proof}
\begin{thm}\label{thm:MisnotI}
Let $\lambda_0\in\mathbb R$ be such that $\mathfrak D(\lambda_0) = \pm 2$, but
$M(\lambda_0) \neq \pm\mathbb I$. Then
\begin{equation*}
\lim_{\lambda\rightarrow\lambda_0} \Gamma(M(\lambda)) = 0.
\end{equation*}
\end{thm}
\begin{proof}
We can write the monodromy matrix as
\begin{equation*}
M(\lambda) = \begin{pmatrix} a(\lambda) & b(\lambda) \\ c(\lambda) & \mathfrak D(\lambda) - a(\lambda) \end{pmatrix} \qquad (\lambda\in\mathbb C)
\end{equation*}
with entire functions $a, b, c$ (and $\mathfrak D$).
As the (Wronskian) $\det M(\lambda) = 1$ for all $\lambda$, we find
\begin{equation}
c b = a \mathfrak D - a^2 - 1.
\label{eq:cb}
\end{equation}
Therefore, if $b(\lambda_0) = c(\lambda_0) = 0$, then
$a(\lambda_0) = \mathfrak D(\lambda_0)/2$ and hence $M(\lambda_0) = \pm\mathbb I$,
contradicting the hypotheses.
So $b(\lambda_0) \neq 0$ or $c(\lambda_0) \neq 0$.

We first consider the case $b(\lambda_0) \neq 0$. Then $b \neq 0$ in a
neighbourhood of $\lambda_0$. For $\lambda$ in this neighbourhood, we can write,
using equation (\ref{eq:cb}),
\begin{equation*}
M(\lambda) = \begin{pmatrix} a(\lambda) & b(\lambda) \\ \frac{a(\lambda)\mathfrak D(\lambda) - a(\lambda)^2 - 1}{b(\lambda)} & \mathfrak D(\lambda) - a(\lambda) \end{pmatrix},
\end{equation*}
with eigenvalues
$\frac{\mathfrak D(\lambda) \pm \sqrt{\mathfrak D(\lambda)^2 - 4}}2$ and
eigenvectors
\begin{equation*}
w_+(\lambda) = \begin{pmatrix} b(\lambda) \\ \frac{\mathfrak D(\lambda) + \sqrt{\mathfrak D(\lambda)^2 - 4}}2 - a(\lambda) \end{pmatrix}, \qquad
w_-(\lambda) = \begin{pmatrix} b(\lambda) \\ \frac{\mathfrak D(\lambda) - \sqrt{\mathfrak D(\lambda)^2 - 4}}2 - a(\lambda) \end{pmatrix}.
\end{equation*}
Then $|w_\pm(\lambda)| \ge |b(\lambda)|$ and, by equation (\ref{eq:Gammadef}),
\begin{equation*}
\Gamma(M(\lambda)) \le \frac{|-b(\lambda) \sqrt{\mathfrak D(\lambda)^2 - 4}|}{|b(\lambda)|^2} = \frac{|\sqrt{\mathfrak D(\lambda)^2 - 4}|}{|b(\lambda)|}
\rightarrow 0 \qquad (\lambda\rightarrow\lambda_0).
\end{equation*}
If $b(\lambda_0) = 0$, then $c \neq 0$ in a neighbourhood of $\lambda_0$, and
for $\lambda$ in this neighbourhood we can write
\begin{equation*}
M(\lambda) = \begin{pmatrix} a(\lambda) & \frac{a(\lambda)\mathfrak D(\lambda) - a(\lambda)^2 - 1}{c(\lambda)} \\ c(\lambda) & \mathfrak D(\lambda) - a(\lambda) \end{pmatrix};
\end{equation*}
this matrix has the same eigenvalues as above and eigenvectors
\begin{equation*}
w_\pm(\lambda) = \begin{pmatrix} \frac{\mathfrak D(\lambda) \pm \sqrt{\mathfrak D(\lambda)^2 - 4}}2 - \mathfrak D(\lambda) + a(\lambda) \\ c(\lambda) \end{pmatrix},
\end{equation*}
so
$|w_\pm(\lambda)| \ge |c(\lambda)|$ and
\begin{equation*}
\Gamma(M(\lambda)) \le \frac{|c(\lambda) \sqrt{\mathfrak D(\lambda)^2 - 4}|}{|c(\lambda)|^2}
= \frac{|\sqrt{\mathfrak D(\lambda)^2 - 4}|}{|c(\lambda)|}
\rightarrow 0 \qquad (\lambda\rightarrow\lambda_0).
\end{equation*}
\end{proof}

\section{Asymptotics for large $\Imag\lambda$}\label{sec:far}
The results of the preceding section show that the functions $\gamma_\pm$ do
not tend to zero at any point in the complex plane and that $\Gamma\circ M$
tends to zero only (and exactly) at the end-points of the spectral bands.
However, they do not yet preclude the possibility that these functions become
small for $\lambda$ far away from the real axis; in fact, the lower bound in
Theorem \ref{thm:gammabound} (a) tends to zero as $|\Imag\lambda| \rightarrow \infty$
and hence is not very good in this respect.
In the present section, we show that in fact $\Gamma(M(\lambda))$ and
$\gamma_\pm(\lambda)$ tend to 1 as $|\Imag\lambda| \rightarrow \infty$, which
implies that the level sets of $\Gamma(M(\lambda)) \gamma_+(\lambda) \gamma_-(\lambda)$ are located in strip neighbourhoods of the real axis.
The basis for this is provided by the following asymptotic of the canonical
fundamental system of a Dirac system with real-valued potential on a bounded
interval; this result may be of interest in its own right.

We focus on the case $\Imag\lambda > 0$, as the asymptotics
for $\Imag\lambda \rightarrow -\infty$ are the same due to the symmetry of
the Dirac equation (\ref{eq:pD}) with real-valued $q$.
\begin{thm}\label{thm:Philim}
Let $q\in L^1[0,a]$ be real-valued and $m \ge 0$. Let
$Q(x) = \int_0^x q$ $(x\in [0,a])$ and let $\mu\in\mathbb R$, $\alpha > 0$.
Then, for each $x\in [0,a]$, the solution of the initial value problem
$(\ref{eq:pD})$ with $\lambda = \mu + i \alpha$ has the asymptotic
for $\alpha\rightarrow\infty$
\begin{align}
e^{-\alpha x}\,\Phi(x,\mu &+ i \alpha) = \frac 12\left(e^{i(Q(x) - \mu x)} + e^{-2\alpha x} e^{-i(Q(x) - \mu x)}\right)\mathbb I
\nonumber
\\
&\ + \frac12\left(-e^{i(Q(x) - \mu x)} + e^{-2\alpha x} e^{-i(Q(x) - \mu x)}\right) \sigma_2
+ O_{\rm unif}({\textstyle \frac 1\alpha}).
\label{eq:Phiasymp}
\end{align}
\end{thm}
\noindent
Here $O_{\rm unif}$ means that the bound is uniform in $x\in [0,a]$.
\begin{cor}\label{cor:Mlim}
Under the hypotheses of Theorem \ref{thm:Philim},
\begin{equation*}
e^{-\alpha x}\,\Phi(x,\mu + i \alpha) = \frac 12\,e^{i(Q(x) - \mu x)}\,(\mathbb I - \sigma_2) + O({\textstyle \frac 1\alpha})
\qquad (\alpha\rightarrow\infty)
\end{equation*}
for each $x \in (0, a]$.
\end{cor}
\begin{proof}[Proof of Theorem \ref{thm:Philim}]
Write $\Phi = \sum\limits_{j=0}^3 \sigma_j\,\phi_j$ with complex-valued
functions $\phi_0$, $\phi_1$, $\phi_2$, $\phi_3$ and the Pauli matrices
\begin{equation}
\sigma_1 = \begin{pmatrix} 0 & 1 \\ 1 & 0 \end{pmatrix}, \quad
\sigma_2 = \begin{pmatrix} 0 & -i \\ i & 0 \end{pmatrix}, \quad
\sigma_3 = \begin{pmatrix} 1 & 0 \\ 0 & -1 \end{pmatrix}, \quad
\sigma_0 =
\begin{pmatrix} 1 & 0 \\ 0 & 1 \end{pmatrix}
= \mathbb I.
\label{eq:Pauli}
\end{equation}
Then the initial value problem (\ref{eq:pD}) is
equivalent to the system
\begin{align}
\begin{pmatrix} \phi_0 \\ \phi_2 \end{pmatrix}' &= (-\alpha - i q + i \mu) \sigma_1 \begin{pmatrix} \phi_0 \\ \phi_2 \end{pmatrix} + m \begin{pmatrix} 1 & 0 \\ 0 & -i \end{pmatrix} \begin{pmatrix} \phi_1 \\ \phi_3 \end{pmatrix},
\nonumber
\\
\begin{pmatrix} \phi_1 \\ \phi_3 \end{pmatrix}' &= (-\alpha - i q + i \mu) (-\sigma_2) \begin{pmatrix} \phi_1 \\ \phi_3 \end{pmatrix} + m \begin{pmatrix} 1 & 0 \\ 0 & i \end{pmatrix} \begin{pmatrix} \phi_0 \\ \phi_2 \end{pmatrix},
\label{eq:coupsys}
\end{align}
with initial values $\phi_0(0) = 1$, $\phi_1(0) = \phi_2(0) = \phi_3(0) = 0$.
We now make the ansatz
\begin{equation}
\begin{pmatrix} \phi_0 \\ \phi_2 \end{pmatrix} = \frac{e^r}2 \begin{pmatrix} 1 \\ 1 \end{pmatrix} u_1 + \frac{e^{-r}}2 \begin{pmatrix} 1 \\ -1 \end{pmatrix} u_2,
\quad
\begin{pmatrix} \phi_1 \\ \phi_3 \end{pmatrix} = \frac{e^r}2 \begin{pmatrix} 1 \\ -i \end{pmatrix} u_3 + \frac{e^{-r}}2 \begin{pmatrix} 1 \\ i \end{pmatrix} u_4,
\label{eq:ansatz}
\end{equation}
with functions $u_1, u_2, u_3, u_4$ and given $r(x) = -\alpha x - i(Q(x) - \mu x)$ $(x\in [0,a])$.
This is motivated by the fact that, with {\it constants} $u_1, u_2, u_3, u_4$,
the above are the general solution of the decoupled equation system when
$m = 0$. The initial conditions translate to
$u_1(0) = u_2(0) = 1$, $u_3(0) = u_4(0) = 0$.
Using this ansatz in the coupled differential equation system (\ref{eq:coupsys}) and then multiplying the first equation (from the left) with the row vectors
$(1, 1)$ and $(1, -1)$, the second equation with the vectors $(1, i)$ and
$(1, -i)$, respectively yields the two separate differential equation systems
\begin{equation}
\left\{\begin{matrix} u_1' = m\,e^{-2r}\,u_4 \\ u_4' = m\,e^{2r}\,u_1 \end{matrix}\right.
\qquad
\left\{\begin{matrix} u_3' = m\,e^{-2r}\,u_2 \\ u_2' = m\,e^{2r}\,u_3 \end{matrix}\right.
\label{eq:sepsys}
\end{equation}
which are in fact the same system, but with different initial values.
Focusing on the system on the left-hand side first, we observe that
\begin{equation*}
|u_1|'(x) \le |u_1'(x)| = m\,e^{2\alpha x}\,|u_4|(x),
\qquad
|u_4|'(x) \le |u_4'(x)| = m\,e^{-2\alpha x}\,|u_1|(x).
\end{equation*}
The solutions exist on $[0,a]$ and, as continuous functions, are bounded.
By an integration by parts and using the initial values,
\begin{align*}
|u_4|(x) &= |u_4(0)| + \int_0^x |u_4|'
\le \int_0^x m\,e^{-2\alpha t}\,|u_1|(t)\,dt
\\
&= -\frac m{2\alpha}\,(e^{-2\alpha x}\,|u_1|(x) - 1) + \frac m{2\alpha} \int_0^x e^{-2\alpha t}\,|u_1|'(t)\,dt
\\
&\le \frac m{2\alpha} - \frac m{2\alpha}\,e^{-2\alpha x}\,|u_1|(x) + \frac {m^2}{2\alpha} \int_0^x e^{-2\alpha t} e^{2\alpha t}\,|u_4|(t)\,d t
\qquad (x\in [0,a]),
\end{align*}
so
\begin{equation*}
\sup_{x\in[0,a]} |u_4(x)| \le \frac m{2\alpha} + \frac {m^2 a}{2\alpha} \sup_{x\in[0,a]} |u_4(x)|
\end{equation*}
and hence
\begin{equation*}
\sup_{x\in[0,a]} |u_4(x)| \le \frac m{2\alpha}\,\frac 1{1 - \frac{m^2 a}{2\alpha}} \le \frac m\alpha
\end{equation*}
for $\alpha > m^2 a$. Consequently,
\begin{equation*}
|u_1(x) - 1| \le \int_0^x |u_1'|
= \int_0^x m\,e^{2\alpha t}\,|u_4(t)|\,dt
\le \frac{m^2}{2\alpha^2}\,(e^{2\alpha x} - 1)
\qquad (x\in[0,a]).
\end{equation*}
Now applying an analogous procedure to the right-hand side system in equation
(\ref{eq:sepsys}), we find
\begin{align*}
|u_2|(x) &= |u_2|(0) + \int_0^x |u_2|'
\le 1 + m \int_0^x e^{-2\alpha t}\,|u_3|(t)\,d t
\\
&= 1 - \frac m{2\alpha}\,(e^{-\alpha x}\,|u_3|(x) - 0) + \frac m{2\alpha}\int_0^x e^{-2\alpha t}\,|u_3|'(t)\,d t
\\
&\le 1 - \frac m{2\alpha}\,e^{-2\alpha x}\,|u_3(x)| + \frac {m^2}{2\alpha}\int_0^x e^{-2\alpha t} e^{2\alpha t}\,|u_2|(t)\,d t,
\end{align*}
so
\begin{equation*}
\sup_{x\in[0,a]} |u_2(x)| \le 1 + \frac{m^2 a}{2\alpha} \sup_{x\in[0,a]} |u_2(x)|
\end{equation*}
and hence
\begin{equation*}
\sup_{x\in[0,a]} |u_2(x)| \le \frac 1{1 - \frac{m^2 a}{2\alpha}} < 2
\end{equation*}
for $\alpha > m^2 a$. Also,
\begin{align*}
|u_3|(x) &= |u_3|(0) + \int_0^x |u_3|'
\le m \int_0^x e^{2\alpha t}\,|u_2|(t)\,d t
\\
&= \frac m{2\alpha}\,(e^{2\alpha x}\,|u_2|(x) - 1) - \frac m{2\alpha} \int_0^x e^{2\alpha t}\,|u_2|'(t)\,d t
\\
&\le \frac m{2\alpha}\,(e^{2\alpha x}\,|u_2|(x) - 1) + \frac{m^2}{2\alpha} \int_0^x e^{2\alpha t} e^{-2\alpha t}\,|u_3|(t)\,d t
\end{align*}
and therefore
\begin{equation*}
\sup_{t\in[0,x]} |u_3(t)| \le \frac m{2\alpha}\,\frac{e^{2\alpha x}}{1 - \frac{m^2 a}{2\alpha}} - \frac m{2\alpha} + \frac{m^2 a}{2\alpha} \sup_{t\in[0,x]} |u_3(t)|,
\end{equation*}
which gives
\begin{equation*}
\sup_{t\in[0,x]} |u_3(t)| \le \frac 1{1 - \frac{m^2 a}{2\alpha}}\left(\frac {m\,e^{2\alpha x}}{2\alpha - m^2 a} - \frac m{2\alpha}\right)
\le \frac{2m}\alpha\,(e^{2\alpha x} - {\textstyle \frac 12})
\le \frac{2m}\alpha\,e^{2\alpha x}
\end{equation*}
for all $x\in [0,a]$ and $\alpha > m^2 a$. Consequently,
\begin{equation*}
|u_2(x) - 1| \le m\int_0^x e^{-2\alpha t}\,|u_3(t)|\,d t
\le \frac{2 m^2 a}\alpha
\qquad (x \in [0,a]).
\end{equation*}
By equation (\ref{eq:ansatz}), we have thus obtained the asymptotics
\begin{align*}
\begin{pmatrix} \phi_0 \\ \phi_2 \end{pmatrix}(x) &= \frac{e^{-\alpha x}}2\,e^{-i(Q(x)-\mu x)} \begin{pmatrix} 1 \\ 1 \end{pmatrix}
(1 + O(\frac{e^{2\alpha x}}{\alpha^2}))
\\ &\qquad
+ \frac{e^{\alpha x}}2\,e^{i(Q(x)-\mu x)} \begin{pmatrix} 1 \\ -1 \end{pmatrix} (1 + O_{\rm unif}({\textstyle\frac 1\alpha})),
\\
\begin{pmatrix} \phi_1 \\ \phi_3 \end{pmatrix}(x) &= \frac{e^{-\alpha x}}2\,e^{-i(Q(x)-\mu x)}\begin{pmatrix} 1 \\ -i \end{pmatrix} O(\frac{e^{2\alpha x}}\alpha)
+ \frac{e^{\alpha x}}2\,e^{i(Q(x)-\mu x)} \begin{pmatrix} 1 \\ i \end{pmatrix} O_{\rm unif}({\textstyle\frac 1\alpha}),
\end{align*}
and equation (\ref{eq:Phiasymp}) follows.
\end{proof}
\noindent
On the basis of the preceding theorem, we now find the asymptotics of the
quasimomentum and of $\Gamma\circ M$ (in Theorem \ref{thm:Masymp}) and of
$\varphi_\pm$ and $\gamma_\pm$ (in Theorem \ref{thm:phipmasymp}).
\begin{thm}\label{thm:Masymp}
Let $\mu\in\mathbb R$, $\alpha > 0$. Then we have the following asymptotics
as $\alpha\rightarrow\infty$ for the monodromy matrix of the periodic
Dirac equation $(\ref{eq:pD})$
\begin{equation*}
e^{-\alpha a}\,M(\mu + i\alpha) = \frac 12\,e^{i(Q(a) - \mu a)} (\mathbb I - \sigma_2) + O({\textstyle \frac 1\alpha}),
\end{equation*}
the discriminant
\begin{equation*}
e^{-\alpha a}\,\mathfrak D(\mu+i \alpha) = e^{i(Q(x) - \mu a)} + O({\textstyle \frac 1\alpha}),
\end{equation*}
the quasimomentum
\begin{equation*}
k(\mu + i \alpha) = \mu + i \alpha - \frac{Q(a)}a + O({\textstyle\frac 1\alpha})
\end{equation*}
and the eigenvectors of $M(\mu + i \alpha)$
\begin{equation*}
v_+(\mu + i \alpha) = \begin{pmatrix} 1 \\ -i \end{pmatrix} + O({\textstyle\frac 1\alpha}),
\qquad
v_-(\mu + i \alpha) = \begin{pmatrix} 1 \\ i \end{pmatrix} + O({\textstyle\frac 1\alpha}).
\end{equation*}
Consequently,
\begin{equation*}
\Gamma(M(\mu + i\alpha)) = 1 + O({\textstyle  \frac 1\alpha}).
\end{equation*}
\end{thm}
\begin{proof}
The asymptotics of the monodromy matrix and hence the discriminant follow
directly from Corollary \ref{cor:Mlim}. As $e^{-\alpha a} M(\mu + i\alpha)$
has determinant $e^{-2\alpha a}$,
solving its characteristic equation
shows that the two eigenvalues of this matrix have asymptotics
$e^{i(Q(a) - \mu a)} + O(\frac 1\alpha)$ and
$e^{-2\alpha a}\,(e^{-i(Q(a) - \mu a)} + O(\frac 1\alpha))$,
respectively. Hence the larger eigenvalue of $M(\mu + i \alpha)$ is
\begin{equation*}
e^{\alpha a}\,(e^{i(Q(a) - \mu a)} + O(\frac 1\alpha))
= e^{-i(\mu + i \alpha - \frac{Q(a)}a + O(\frac 1\alpha))a},
\end{equation*}
and we can read off the asymptotics for the quasimomentum. The
asymptotic form of the eigenvectors follows from that of the matrix
$e^{-\alpha a} M(\mu + i \alpha)$ and its eigenvalues.
\end{proof}
\begin{thm}\label{thm:phipmasymp}
Let $\mu \in \mathbb R$, $\alpha > 0$.
For the periodic Dirac equation $(\ref{eq:pD})$, the periodic functions
$\varphi_\pm$ of equation $(\ref{eq:Flop})$ have asymptotics
\begin{equation*}
|\varphi_\pm(x,\mu + i \alpha)| = |\varphi_\pm(0,\mu + i \alpha)|\,(1 + O_{\rm unif}({\textstyle\frac 1\alpha}))
\qquad
 (\alpha \rightarrow \infty)
\end{equation*}
uniformly in $x\in [0,a]$.
Consequently, the functions $\gamma_\pm$ of equation $(\ref{eq:gammadef})$
satisfy
\begin{equation*}
 \gamma_\pm(\mu + i \alpha) = 1 + O({\textstyle \frac 1 \alpha})
\qquad (\alpha \rightarrow \infty).
\end{equation*}
\end{thm}
\begin{proof}
By Theorem \ref{thm:Masymp} and equation (\ref{eq:Phiasymp}),
observing that
\begin{equation*}
{\textstyle\frac 12}\,(\mathbb I - \sigma_2)\,v_+(\mu + i\alpha) = v_+(\mu + i\alpha) + O({\textstyle\frac 1\alpha}),
\quad
{\textstyle\frac 12}\,(\mathbb I + \sigma_2)\,v_+(\mu + i\alpha) = O({\textstyle\frac 1\alpha}),
\end{equation*}
we obtain
\begin{align*}
\varphi_+(x,\mu+i\alpha) &= e^{i k(\mu+i\alpha) x} \Phi(x,\mu+i\alpha)\,v_+(\mu+i\alpha)
\\
&= e^{i(\mu - \frac{Q(a)}a + O(\frac 1\alpha)) x}\,e^{-\alpha x}\, \Phi(x,\mu+i\alpha)\,v_+(\mu+i\alpha)
\\
&= e^{i(Q(x) - \frac{Q(a)}a\,x + O_{\rm unif}({\frac 1\alpha}))}\,(v_+(\mu+i\alpha) + O({\textstyle\frac 1\alpha}))
\\
&\qquad
+ e^{-2\alpha x} e^{-i(Q(x) + \frac{Q(a)}a\,x - 2 \mu x + O_{\rm unif}({\frac 1\alpha}))}\,O({\textstyle\frac 1\alpha}) + O_{\rm unif}({\textstyle\frac 1\alpha})
\end{align*}
and hence
\begin{align*}
|\varphi_+(x,\mu+i\alpha)| &= (1 + O_{\rm unif}({\textstyle\frac 1\alpha}))\,(v_+(\mu+i\alpha) + O({\textstyle\frac 1\alpha})) + O_{\rm unif}({\textstyle\frac 1\alpha})
\\
&= |v_+(\mu+i\alpha)| + O_{\rm unif}({\textstyle\frac 1\alpha})
= |\varphi_+(0,\mu+i\alpha)|(1 + O_{\rm unif}({\textstyle\frac 1\alpha})).
\end{align*}
Analogous reasoning for $\varphi_-(x,\mu+i\alpha)$ does not work since the
exponentially large factor $e^{-ik(\mu+i\alpha)x}$ ($=\rho(\lambda)^{x/a}$ in equation (\ref{eq:Flop})) leads to uncontrolled
amplification of the $O({\textstyle\frac 1\alpha})$ error term.
However, the Floquet solution $u_-(x,\lambda)$ is equal, up to a constant factor, to
$\sigma_3 \tilde u_+(a-x,\lambda)$ $(x\in\mathbb R)$, where $\tilde u_+$ is the
Floquet solution corresponding to the eigenvalue of modulus greater than 1
of the periodic Dirac equation with potential $\tilde q(x) := q(a-x)$ $(x\in
\mathbb R)$. Therefore the corresponding periodic functions satisfy
(up to a constant factor)
$|\varphi_-(x,\lambda)| = |\tilde \varphi_+(a-x,\lambda)|$, so
we obtain the asymptotics of $|\varphi_-|$ by applying the above reasoning to
$|\tilde \varphi_+|$.
\end{proof}

\section*{Acknowledgements}
Ghada Shuker Jameel expresses her gratitude to the University of Mosul (Iraq) for
the financial support towards her PhD studies in the UK.

\end{document}